\documentclass[12pt]{article}
\usepackage[margin=1in]{geometry}
\usepackage{url}
\usepackage{graphicx}
\usepackage{subcaption}

\usepackage[T1]{fontenc}

\usepackage{amsmath}
\usepackage{amssymb}
\usepackage{amsthm}
\usepackage{mathrsfs}

\theoremstyle{plain}
\newtheorem{theorem}{Theorem}

\newtheorem{lemma}[theorem]{Lemma}

\theoremstyle{definition}

\theoremstyle{remark}

\newcommand{\VTM}{\mbox{\hspace{.01in}{\tt vtm}}}
\newcommand{\Thue}{\mbox{\hspace{.01in}{\tt vtm}}}

\title{Squarefree words with interior disposable factors}
\author{
Marko Milosevic and Narad Rampersad\thanks{The second author is supported by
  NSERC Discovery Grant 2019-04111.}\\
Department of Mathematics and Statistics\\
University of Winnipeg\\
\url{n.rampersad@uwinnipeg.ca}
}
\date{\today}

\begin{document}
\maketitle

\begin{abstract}
We give a partial answer to a problem of Harju by constructing an infinite
ternary squarefree word $w$ with the property that for every $k \geq 3312$
there is an interior length-$k$ factor of $w$ that can be deleted while still
preserving squarefreeness.  We also examine Thue's famous
squarefree word (generated by iterating the map $0 \to 012$, $1 \to
02$, $2 \to 1$) and characterize the positions $i$ for which deleting
the symbol appearing at position $i$ preserves squarefreeness.
\end{abstract}

\section{Introduction}
The study of squarefree words (words avoiding non-empty repetitions
$xx$) is a fundamental topic in combinatorics on words.  Thue
\cite{Thu06} was the first to construct an infinite squarefree word on
three symbols.  Recently, Harju \cite{Har20} defined an interesting
class of squarefree word: \emph{irreducibly squarefree words}.  A
squarefree word is irreducibly squarefree if the deletion of any
letter in the word, other than the first and last letters, produces an
occurrence of a square.  Harju showed that there exist ternary
irreducibly squarefree words of all sufficiently large lengths.
Harju's notion of irreducibly squarefree words was inspired by a
similar concept introduced by Grytczuk, Kordulewski, and Niewiadomski
\cite{GKN20}, who defined \emph{extremal squarefree words} as follows:
a squarefree word is extremal if every possible insertion of a symbol
into the word creates an occurrence of a square.

Harju posed three open problems in his paper.  We give a partial
answer to his third problem here by constructing an infinite
squarefree word $w$ with the property that for every $k \geq 3312$
there is an interior (i.e., not a prefix) length-$k$ factor of $w$
that can be deleted while still preserving squarefreeness.  We also
examine Thue's famous squarefree word (generated by iterating the map
$0 \to 012$, $1 \to 02$, $2 \to 1$) \cite{Thu12} and characterize the
positions $i$ for which deleting the symbol appearing at position $i$
preserves squarefreeness.

\section{Preliminaries}
Let $A$ be a finite alphabet of letters.
For a word $w$ over $A$ (i.e., $w \in A^*$), let $|w|$ denote its length.
A word $u$ is a \emph{factor} of $w$, if $w=xuy$ where $x$ and/or $y$ may be empty.
If $x$ ($y$, resp.) is empty then $u$ is a \emph{prefix} (a \emph{suffix}, resp.)  of $w$.

A \emph{square} is a non-empty word of the form $u^2 = uu$.  A finite
or infinite word $w$ over $A$ is \emph{squarefree} if it does not have
any square factors.  A position $i$ in a squarefree word $w$ is said
to be \emph{disposable} if $w = uav$, $a \in A$, $|u|=i$, and the word
$uv$ is squarefree.  If $w=uxv$ and $uv$ are squarefree, then $x$ is a
\emph{disposable factor} of $w$.

A morphism $h\colon A^* \to A^*$ is said to be \emph{squarefree},
if it preserves squarefreeness of words, i.e., if $h(w)$ is squarefree for all squarefree words $w$.
A morphism $h\colon A^* \to A^*$ is \emph{uniform} if the images $h(a)$ have the same length:
$|h(a)|=n$ for all $a \in A$ and for some positive $n$ called the \emph{length} of~$h$.

An infinite word $w$ is a \emph{fixed point} of a morphism $h$ if $h(w)=w$. This happens
if $w$ begins with the letter $a$, and $w$ is obtained by iterating $h$ on
the first letter~$a$ of $w$: $h(a)=au$ and
$w=auh(u)h^2(u)\cdots$. In this case we denote the fixed point $w$ by~$h^\omega(a)$.

Let $T$ be the ternary alphabet $T = \{0,1,2\}$.
Let $\tau\colon T^* \to T^*$ be the morphism defined by
\begin{equation}\label{Hall-Thue}
\tau(0)=012, \  \tau(1)=02 \  \text{and} \ \tau(2)=1\,.
\end{equation}
The  word obtained by iterating $\tau$ on $0$ gives the following
infinite squarefree word:
\[
\Thue=012 02 1 012 1 02 012 \cdots \ (=\tau^\omega(0)).
\]
(Here we follow~\cite{BSCFR14} in using $\VTM$, for \emph{variant
  of the Thue--Morse word}, to denote this word.)
For the next basic result, see~\cite{Braunholtz,Hal64,Istrail,
MorseHedlund,Thu12} and~\cite{Lothaire}:

\begin{lemma}\label{lemma:Hall_no}
The word $\Thue$ is squarefree  and it does not contain $010$ or $212$ as factors.
\end{lemma}

Note that this lemma implies that $\Thue$ does not contain $1021$ or
$1201$ as factors, since the only way these could arise are as factors of
$\tau(212) = 1021$ or $\tau(00)= 012012$.

\section{Disposable positions in $\Thue$}

Here is a list of the first few disposable positions in $\Thue$:
\[
  (0, 2, 12, 18, 44, 50, 60, 66, 76, 82, 108, 114, 140, 146, 172, 178,
  188, 194, 204, \ldots)
\]
and here is a list of the first few first differences of the above sequence:
\[
  (2, 10, 6, 26, 6, 10, 6, 10, 6, 26, 6, 26, 6, 26, 6, 10, 6, 10, 6,
  26, 6, 10, 6, 10, 6, 26, \ldots).
\]
The goal of this section is to give an exact description of these two
sequences.

The letter in position $0$ of $\VTM$ is trivially disposable.  The
disposability of the $2$ in position $2$ of $\Thue$ is easy to verify.
Deleting this letter produces a new word $01w'$, where $w'$ is a
squarefree suffix of $\VTM$. If a square occurs, it must begin from
the first letter $0$, or the second letter $1$. In either case, a
factor $1021$ or $010$ is found in the first half of the
square. Since $\VTM$ avoids both factors, this leads to a
contradiction, so we conclude that this occurrence of 2 is disposable.

\begin{theorem}\label{dispo_0}
The second and fourth occurrences of $0$ in a factor
$\tau(10121) = 02\mathbf{0}12021\mathbf{0}2$ of $\VTM$ are
disposable in $\VTM$.
\end{theorem}

\begin{proof}
Every disposable letter divides a squarefree word
into a squarefree prefix consisting of all letters before the
disposable one, and a squarefree suffix consisting of all letters
afterwards. For the remainder of this proof, let $u$ denote the
aforementioned prefix, and $v$ the corresponding suffix. Then $uv$ is
the word obtained by deleting the disposable letter.

We first notice that every factor $\tau(10121)$ is enclosed by
$\tau(02)$. That is, $\tau(10121)$ always occurs in the context
$$\tau(02)\tau(10121)\tau(02) =
012102\mathbf{0}12021\textbf{0}20121.$$ After deleting the second $0$ in
$\tau(10121)$, we have $uv = w'01210212021w''$, where $u = w'012102$
and $v = 12021w''$. Since any potential square would have to start in
$u$ and end in $v$, to ensure that every occurrence of this $0$ is
disposable, we must verify that, no matter which letter in $u$ we
start with, no square occurs.

It is clear that $01210212021$ (and hence $uv$) is squarefree for all
squares of length at most six, and any square in $uv$ of length at
least eight that crosses the boundary between $u$ and $v$ contains a
factor $212$ or $1021$ in each half of the square, which is a
contradiction, since $\VTM$ avoids $212$ and $1021$. The only possible
exception is a square of the form $xx = 1y1021y102$, where $y \in T^*$
and $x = 1y102$ is both a prefix of $v$ and a suffix of $u$.  Now, the
square $xx$ was produced by deleting the 0 from $x0x$ which implies
that 0 does not immediately follow the square. The only other option
is $1$. This gives us $xx1 = 1y1021y1021$; a contradiction again since
$\VTM$ avoids $1021$.

The argument for the fourth 0 in $\tau(10121)$ is similar.
\end{proof}

\begin{theorem}\label{dispo_2}
The first and third occurrences of $2$ in a factor $\tau(12101) =
0\mathbf{2}10201\mathbf{2}02$ of $\VTM$ are disposable in $\VTM$.
\end{theorem}

\begin{proof}
  We proceed in a similar manner to Theorem~\ref{dispo_0}. Let $uv$ be
  as described in the proof of that theorem.  First, we
  have that $\tau(12101)$ only occurs in the context
  $$\tau(20)\tau(12101)\tau(20) =
  10120\mathbf{2}10201\mathbf{2}021012.$$ Deleting the first 2 in
  $\tau(12101)$ gives us $uv = w'10120102012w''$, where $u = w'10120$
  and $v = 102012w''$. Indeed, the factor $10120102012$ (and hence
  $uv$) avoids squares of length at most six. So any potential square
  in $uv$ must have length at least eight, but then it contains $010$
  or $1201$ in each half of the square, which is not possible, since
  $\VTM$ does not contain $010$ or $1201$. However, we must consider
  the exception $xx = 1y1201y120$, where $y \in T^*$ and $x=1y120$ is
  a prefix of $v$ and a suffix of $u$.  Then the deletion of $2$ from
  $x2x$ implies that a $1$ must occur immediately after the square,
  giving us $xx1 = 1y1201y1201$, which is a contradiction, since
  $\Thue$ does not contain $1201$.

  The argument for the third 2 in $\tau(12101)$ is similar.
\end{proof}

Theorems~\ref{dispo_0} and \ref{dispo_2} only show that certain
positions in $\Thue$ are disposable and do not indicate which
positions are not disposable.  We can completely characterize the
disposable positions in $\Thue$ using the computer program Walnut
\cite{Walnut}.

The word $\Thue$ is a $2$-automatic sequence (see \cite{Ber79}) and is
generated by the automaton in Figure~\ref{vtm_aut} (the labels of each
state indicate the output associated with that state).

  \begin{figure}[htb]
    \centering
    \includegraphics[scale=0.75]{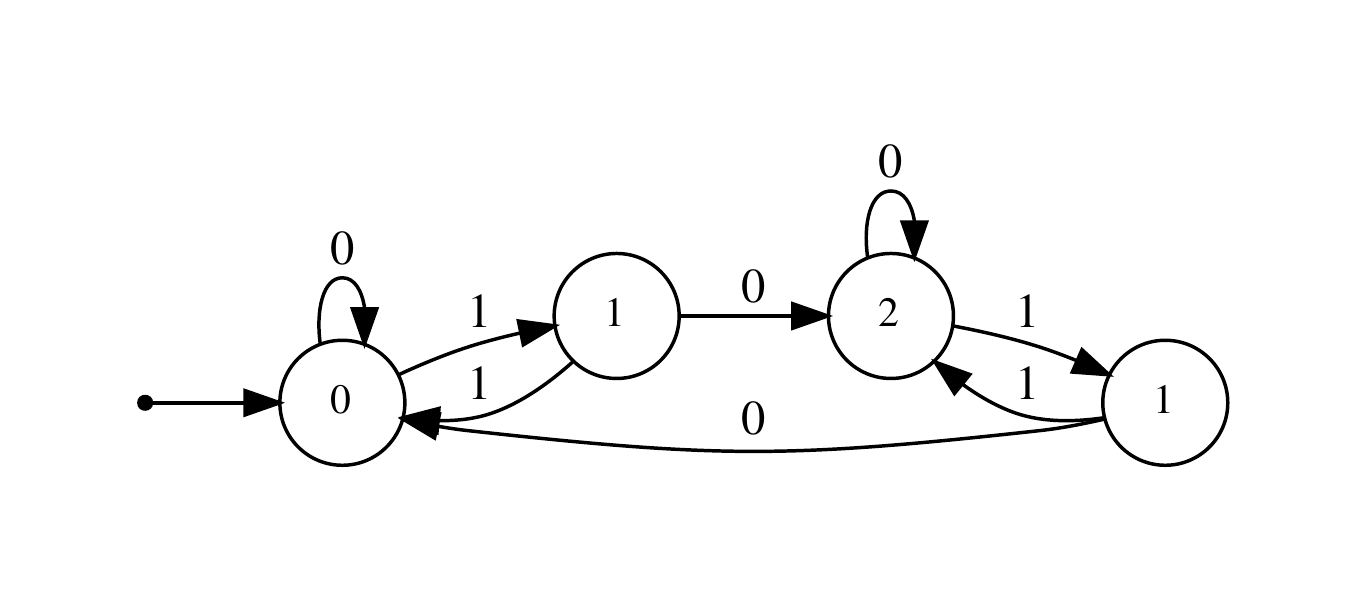}
    \caption{$2$-DFAO for $\VTM$}\label{vtm_aut}
  \end{figure}

  Since $\VTM$ is an automatic sequence, we can use
  Walnut~\cite{Walnut} to verify that it has certain combinatorial
  properties.  The following Walnut command computes the set of disposable
  positions $j$ in $\Thue$.  The output automaton is given in
  Figure~\ref{dispo_pos}.

  \begin{verbatim}
  eval dispo_pos "?msd_2 Ai,n (i < j & j < i+2*n) => (Ek i <=
    k & ((j < i+n & k <= i+n) | (j >= i+n & k < i+n)) & (((j < k | j >
    k+n) & VTM[k] != VTM[k+n]) | ((k < j & j <= k+n) & VTM[k] !=
    VTM[k+n+1])))";
  \end{verbatim}
  
  \begin{figure}[htb]
    \centering
    \includegraphics[scale=0.5]{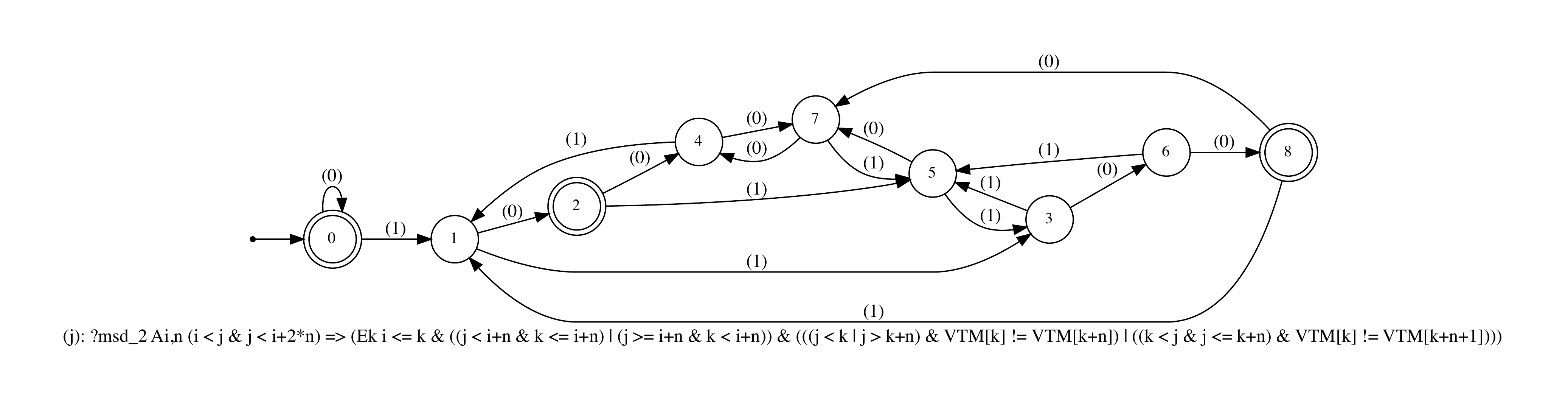}
    \caption{\texttt{dispo\_pos} output automaton}\label{dispo_pos}
  \end{figure}

  The next command computes the set of values taken by the \emph{first
    difference} sequence of the sequence of disposable positions in
  $\Thue$ (excluding the initial position).
  The output automaton is given in Figure~\ref{dispo_delta}.
  
   \begin{verbatim}
  eval dispo_delta "?msd_2 Ei,j i >=2 & j > i & j = i+l &
  $dispo_pos(i) & $dispo_pos(j) & (Ak (i<k & k<j) =>
  ~$dispo_pos(k))";
  \end{verbatim}

  \begin{figure}[htb]
    \centering
    \includegraphics[scale=0.75]{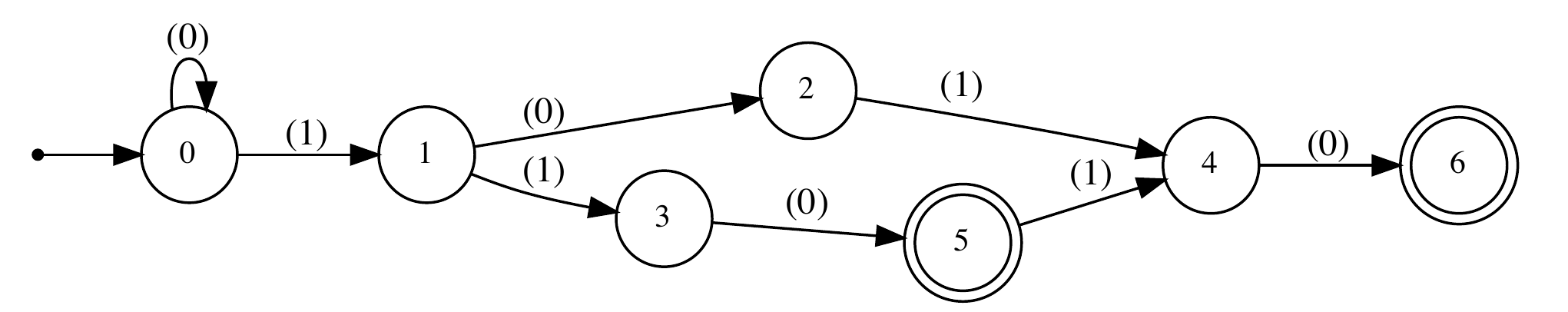}
    \caption{\texttt{dispo\_delta} output automaton}\label{dispo_delta}
  \end{figure}

  From Figure~\ref{dispo_delta}, we see that the ``gaps'' between
  disposable positions (excluding the initial position) in $\Thue$ are
  $6$, $10$, and $26$.  We next consider the \emph{density} of the
  disposable positions in $\VTM$.  Let $D_w(n)$ denote the set of
  disposable positions $\leq n+1$ of an infinite squarefree word $w$.
  Figure~\ref{dispo_plot} shows a plot of the initial values of
  $|D_{\VTM}(n)|/n$.

    \begin{figure}[htb]
    \centering
    \includegraphics[scale=1]{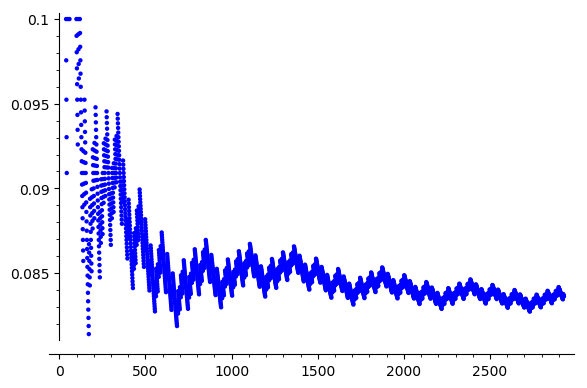}
    \caption{Plot of $|D_{\VTM}(n)|/n$}\label{dispo_plot}
  \end{figure}

  We would like to determine the quantity
  $\lim_{n \to \infty} |D_{\VTM}(n)|/n$ if it exists, or failing that,
  the quantities $\liminf_{n \to \infty} |D_{\VTM}(n)|/n$ and
  $\limsup_{n \to \infty} |D_{\VTM}(n)|/n$.  There is a general method
  for this due to Bell \cite{Bell20}; however, due to the structure of
  the automaton in Figure~\ref{dispo_pos}, we are able to employ
  simpler techniques.

  \begin{theorem}
    The density of disposable positions in $\VTM$ is
    \[ \lim_{n \to \infty} |D_{\VTM}(n)|/n = 1/12 = 0.08\dot{3}.\]
  \end{theorem}

  \begin{proof}
    Let $M$ be the adjacency matrix of the automaton given in
    Figure~\ref{dispo_pos}, restricted to just the states $1$ to $8$.
    That is, let $M$ be the $8 \times 8$ matrix whose $ij$-entry is
    equal to the number of transitions from state $i$ to state $j$.
    We have
    \[ M=
    \left(\begin{array}{rrrrrrrr}
0 & 1 & 1 & 0 & 0 & 0 & 0 & 0 \\
0 & 0 & 0 & 1 & 1 & 0 & 0 & 0 \\
0 & 0 & 0 & 0 & 1 & 1 & 0 & 0 \\
1 & 0 & 0 & 0 & 0 & 0 & 1 & 0 \\
0 & 0 & 1 & 0 & 0 & 0 & 1 & 0 \\
0 & 0 & 0 & 0 & 1 & 0 & 0 & 1 \\
0 & 0 & 0 & 1 & 1 & 0 & 0 & 0 \\
1 & 0 & 0 & 0 & 0 & 0 & 1 & 0
          \end{array}\right).
      \]      

    We can verify that $M^5$ is a positive matrix, which implies that
    $M$ is a \emph{primitive} matrix.  The Perron--Frobenius
    eigenvalue of $M$ is $2$, and it follows from the standard
    Perron--Frobenius theory (see \cite[Chapter~8]{AS03} for a
    treatment formulated in terms of morphisms rather than automata)
    that the $i$-th entry of the left eigenvector
    \[ {\bf v} = (1/12, 1/24, 1/6, 1/8, 1/4, 1/12, 5/24, 1/24) \] of
    $M$, where we have normalized ${\bf v}$ so that its entries sum to
    $1$, gives the fraction of all input strings that reach state $i$.
    Since the final states of the automaton in Figure~\ref{dispo_pos}
    are $2$ and $8$, it follows that the fraction of all input strings
    that reach a final state is the sum of the entries of ${\bf v}$ in
    positions $2$ and $8$.  We conclude that the frequency of
    disposable positions in $\VTM$ is
    $1/24 + 1/24 = 1/12 = 0.08\dot{3}$.
    \end{proof}
    
\section{Words with longer interior disposable factors}

Harju \cite[Problem~3]{Har20} asked if there exists an infinite
ternary squarefree word containing interior disposable factors of every length
$k \geq 1$.  We are unable to prove that this is the case, but we can
prove the following weaker result.

\begin{theorem}\label{dispo_3312}
  There exists an infinite ternary squarefree word containing interior
  disposable factors of every length $k \geq 3312$.
\end{theorem}

\begin{proof}
  Let $A = \{0,1\}$, $B = \{2,3,4\}$, and $C = A \cup B$.
Fraenkel and Simpson \cite{FS95} showed that there exists an infinite
word over $A$ whose only square factors are $00$, $11$, and
$0101$ (see \cite{HN06} for a simpler construction).  Let $x$ be any
such infinite word and let $y$ be any infinite squarefree word over
$B$.  We are going to start by constructing an infinite squarefree word
$v$ over $C$ by interleaving the words $x$ and $y$ in a very
particular way.

For $n \geq 1$, let $Q(n)$ denote any word of the
form $aY_1bY_2cY_3d$ where
\begin{itemize}
\item $a,b,c,d \in A$,
\item $b \neq c$,
\item $abcd \neq 0101$,
\item $Y_i \in B^*$, and
\item $|Y_i| = n$.
\end{itemize}
We now construct $v$ as follows:
\begin{itemize}
\item Taking one symbol at a time, interleave $x$ and $y$ in the following way,
$$x_0y_0x_1y_1x_2y_2\cdots x_iy_ix_{i+1}y_{i+1}x_{i+2}y_{i+2}x_{i+3}\cdots$$
until we get $414$ occurrences of $Q(1)$.
\item Then start taking two symbols at a time from $y$, so at a certain point we have
$$\cdots x_jy_jy_{j+1}x_{j+1}y_{j+2}y_{j+3}x_{j+2}y_{j+4}y_{j+5}x_{j+3}\cdots$$
until we get $414$ occurrences of $Q(2)$.
\item Then take three symbols at a time from $y$:
$$\cdots x_py_qy_{q+1}y_{q+2}x_{p+1}y_{q+3}y_{q+4}y_{q+5}x_{p+2}\cdots$$
until we get $414$ occurrences of $Q(3)$, etc.
\end{itemize}
This gives an infinite squarefree word
\begin{equation*}
    \begin{split}
        v = x_0y_0x_1y_1x_2y_2\cdots & x_iy_ix_{i+1}y_{i+1}x_{i+2}y_{i+2}x_{i+3}\cdots\\\cdots x_jy_jy_{j+1}&x_{j+1}y_{j+2}y_{j+3}x_{j+2}y_{j+4}y_{j+5}x_{j+3}\cdots\\\cdots &x_py_qy_{q+1}y_{q+2}x_{p+1}y_{q+3}y_{q+4}y_{q+5}x_{p+2}\cdots
    \end{split}
  \end{equation*}
over the alphabet $C$. The word $v$ therefore has the form
$$v = \cdots Q(1)\cdots Q(1)\ \cdots\  Q(2)\cdots Q(2)\ \cdots\ Q(n)\cdots Q(n)\cdots$$
where, for each $n$, there are $414$ occurrences of $Q(n)$.  Note that
although we have shown the $Q(n)$ above as being non-overlapping, they
may indeed overlap each other; this poses no problems for the argument
below.

Since $y$ is squarefree, the word $v$ is as well.  Furthermore, we
claim that for any $n \geq 1$, the factor $Y_2$ in any given
$Q(n) = aY_1b\ Y_2\ cY_3d$ can be removed and the resulting word
$\hat{v}$ is squarefree.  To see this, suppose to the contrary that
$\hat{v}$ contains a square $uu$. Since $bc$ does not occur anywhere
else in $\hat{v}$, the first $u$ must end at the $b$ in $aY_1bcY_3d$
and the second $u$ must begin at the $c$.  Since $b \neq c$, the word
$u$ must contain at least two letters from $A$.  If $u$ contains
exactly two letters from $A$, then $ abcd$ is a square in $x$.
However, the only square of length $4$ in $x$ is $0101$, and this
contradicts the hypothesis that $abcd \neq 0101$.  If $u$ contains
more than two letters from $A$, then $x$ contains a square of length
$\geq 6$, which is a contradiction.

Let $h_5 : C^* \rightarrow \{0,1,2\}^*$ be the following $18$-uniform
morphism given by Brandenburg \cite[Theorem~4]{Brandenburg}.
\begin{equation*}
h_5:\left. \ \ \begin{aligned}
        & 0 \rightarrow 010201202101210212 \\
        & 1 \rightarrow 010201202102010212 \\
        & 2 \rightarrow 010201202120121012 \\
        & 3 \rightarrow 010201210201021012 \\
        & 4 \rightarrow 010201210212021012 \\
    \end{aligned}
    \right.
\end{equation*}
Brandenburg proved that $h_5$ is a squarefree morphism; i.e., it maps
squarefree words to squarefree words.  Let $v' = h_5(v)$.  The
infinite word $v'$ is squarefree, and furthermore, since $h_5$ is a
squarefree morphism, we see that any factor $h_5(Y_2)$
occurring in the context $h_5(Q(n)) = h_5(aY_1b\ Y_2\ cY_3d)$ can be
deleted from $v'$ and the resulting word is also squarefree.  Note
that $|h_5(Y_2)| = 18n$.

The next step in the construction is to apply the following squarefree
multi-valued morphism $g$ from \cite[Theorem~20]{CHOR19}:
\begin{align*}
g(0)=&
\begin{cases}
012102120210\mathbf{12}021201210\\
012102120210\mathbf{201}021201210\\
012102120210\mathbf{2012}021201210\\
012102120210\mathbf{20121}021201210,
\end{cases}\\
\end{align*}
$g(1)=\pi(g(0))$, and $g(2)=\pi(g(1))$, where $\pi$ is the permutation
$(0\, 1\, 2)$.  Note that the images of each
letter have lengths $23$, $24$, $25$, or $26$.  Consequently,
for any $a \in C$, the words in the set $g(h_5(a))$ all have
lengths between $18 \times 23 = 414$ and $18 \times 26 = 414+54$ and
all such lengths are obtained.

Let $w$ be the infinite word in $g(v')$ obtained as follows.  When
applying $g$ to $v'$, in general, we choose to replace each letter of
$v'$ with its image under $g$ of length $23$, except in the following
situation.

For each $n \geq 1$ and each $i \in \{0,\ldots,\min\{413, 54n\}\}$,
when applying $g$ to the $i$-th occurrence of $h_5(Q(n)) =
h_5(aY_1bY_2cY_3d)$, replace $h_5(Y_2)$ with any word $Z$ in $g(h_5(Y_2))$
of length $414n+i$.  Since $g$ is a squarefree multi-valued morphism,
the word $Z$ is a disposable factor of $w$ of length
$414n+i$.  The set of lengths of all such disposable factors is
therefore
\[
  L = \bigcup_{n \geq 1} \{414n + i : i \leq \min\{413, 54n\}\} =
  \{k : k \geq 3312\} \cup \bigcup_{n=1}^7 \{414n + i : i \leq 54n\}.
\]
This completes the proof; we note in conclusion that there are $1792$
lengths missing from the set $L$.
\end{proof}

Some of the missing lengths from the proof of Theorem~\ref{dispo_3312}
can be obtained by the following observation due to Harju: if $w$ is
an infinite squarefree word and there is some factor $p$ and some
letter $a$ for which we
can write $w = apaw'$, then $pa$ is disposable, since the resulting
word $aw'$ is a suffix of $w$ and hence is squarefree.  Thus, for
every such $p$, we can add the length $|pa|$ to the list of lengths of
disposable factors of $w$.  Note that to explicitly calculate these
additional lengths for a word $w$ constructed as described in the
proof of Theorem~\ref{dispo_3312}, we would have to make some explicit
choices for the word $x$ and the word $y$ used in the proof, as well
as an explicit rule for choosing the word $Z$ in $g(h_5(Y_2))$.

\section{Conclusion}
An obvious open problem is to completely resolve Harju's question by
improving the construction of Theorem~\ref{dispo_3312} so that there
are interior disposable factors of every length.  Harju \cite{Har20} also
stated two other very interesting open problems in his paper, which we
have not been able to solve.

Regarding the disposable positions in $\VTM$, the main property of
$\VTM$ that allowed us to identify the disposable positions was that
$\VTM$ avoids $010$ and $212$.  This places it in one of the three
classes of squarefree words characterized by Thue \cite{Thu12}: 1)
those avoiding $010$ and $212$; 2) those avoiding $010$ and $020$;
and, 3) those avoiding $121$ and $212$.  It might be interesting to
study disposable positions in words from classes 2) and 3).

We also found that the set of disposable positions in $\VTM$ is fairly
\emph{dense}: the density of disposable positions in $\VTM$ is $1/12$.
Let $D_w(n)$ denote the set of disposable positions $\leq n+1$ of an
infinite squarefree word $w$.  What is the greatest possible value of
$\liminf_{n \to \infty} |D_w(n)|/n$ over all infinite ternary
squarefree words $w$?  Is it achieved by $\VTM$?

\end{document}